\documentclass[12pt]{article}
\usepackage{lineno,hyperref}
\modulolinenumbers[1]
\usepackage{bbm}
\usepackage{float}

\usepackage{mathrsfs}
\usepackage{graphicx}

\usepackage[utf8]{inputenc} 

\usepackage{amsmath}
\usepackage{amstext}
\usepackage{amsfonts}
\usepackage{amsthm}
\usepackage{amssymb}
\usepackage[bf,SL,BF]{subfigure}
\usepackage{subfigure}
\usepackage{caption}

\hypersetup{colorlinks}

\usepackage{color}

\newcommand{\R}{{\mathbb{R}}}

\newcommand{\E}{{\mathbb{E}}}

\renewcommand{\P}{{\mathbb{P}}} 

\newcommand{\triple}{{\vert\kern-0.25ex\vert\kern-0.25ex\vert}}

\theoremstyle{plain}

\newtheorem{definition}{Definition}[section]
\newtheorem{theorem}[definition]{Theorem}
\newtheorem{lemma}[definition]{Lemma}
\newtheorem{corollary}[definition]{Corollary}

\newtheorem{assumption}[definition]{Assumption}

\theoremstyle{definition}
\newtheorem{remark}[definition]{Remark}

\newtheorem{example}[definition]{Example}

\begin{document}
\title{\bf Equivalence of pth moment stability between stochastic differential delay equations and their numerical methods}

\author
{{\bf  Zhenyu Bao, Jingwen Tang, Yan Shen, Wei Liu\footnote{Corresponding author, Email: weiliu@shnu.edu.cn, lwbvb@hotmail.com}}
\\
Department of Mathematics, \\ Shanghai Normal University, Shanghai, 200234, China
}

\date{}

\maketitle

\begin{abstract}
In this paper, a general theorem on the equivalence of pth moment stability between stochastic differential delay equations (SDDEs) and their numerical methods is proved under the assumptions that the numerical methods are strongly convergent and have the bouneded $p$th moment in the finite time. The truncated Euler-Maruyama (EM) method is studied as an example to illustrate that the theorem indeed covers a large ranges of SDDEs. Alongside the investigation of the truncated EM method, the requirements on the step size of the method are significantly released compared with the work, where the method was initially proposed.

\medskip \noindent
{\small\bf Key words}: stochastic differential delay equations, equivalence of pth moment stability, truncated Euler-Maruyama mehtod, highly non-linear coefficients.
\end{abstract}

\section{introduction}
Higham, Mao and Stuart in \cite{HMS2003} initialised the study on the equivalence of stabilities between solutions of stochastic differential equations (SDEs) and their numerical solutions. To be more precise, it is proved in their paper that underlying solutions are mean square stable if and only if numerical solutions are also stable in the mean square sense. The result applies under the assumption of the finite time convergence of the numerical methods.
\par
Mao in \cite{Mao2007_JCAM} extended such a result to stochastic differential delay equations (SDDEs). Zhao, Song and Liu in \cite{ZSL2009_IJNAM} investigated such equivalence for SDDEs with Poisson jump and Markov switching. More recently, Liu, Li and Deng in \cite{LLD2018_ANM} studied neutral delayed stochastic differential equations for such a problem. Deng et la. in \cite{DFFM2019} extended results in \cite{Mao2007_JCAM} to SDDEs driven by $G$-Brownian motion. All the works mentioned above were devoted to the exponential stability in the mean square sense.
\par
Mao in \cite{Mao2015_SIAM} generalised the results in \cite{HMS2003} to the case of pth moments and made some connections of the almost sure stability between SDEs and their numerical solutions. Yang and Li in \cite{YL2019} discussed similar problems in the $G$-framework.
\par
In this paper, we study the equivalence of pth moment stability between solutions of SDDEs and their numerical solutions, which could be regarded as a generalisation of \cite{Mao2007_JCAM}. In addition, we investigate the truncated Euler-Maruyama (EM) method as an example. Compared with those classical Euler-type methods discussed in previous works mentioned above, we do not need to impose the global Lipschitz condition on either the drift or diffusion coefficient.
\par
The truncated EM method was proposed originally by Mao in \cite{Mao2015,Mao2016_JCAM} for SDEs. After that, plenty of works that employed the truncating technique have been done for SDDEs. For example, the truncated EM method was studied by Guo, Mao and Yue in \cite{GMY2018}, the partially truncated EM method was investigated by Zhang, Song and Liu in \cite{ZSL2018}, and the truncated Milstein method was discussed by Zhang, Yin, Song and Liu in \cite{ZYSL2019}.
\par
Our theorem on the truncated EM for SDDEs in this paper is of interest in two aspects. Firstly, the result of the truncated EM demonstrates Theorem \ref{t1} indeed covers a large class of SDDEs and numerical methods. Secondly, the requirement on the step size of the method is significantly released compared with the existing works, which is a stand-alone interesting result. It should be mentioned that many other interesting numerical methods have been proposed for SDDEs, for example \cite{Akhtari2019,Buc2000,CZK2015,CW2015,HH2018,Huang2014,LXW2019,KS2014,QZ2019,WGW2015,ZH2017} and the references therein.
\par
The main contributions of this paper are summarized as follows.
\begin{itemize}
\item A general theorem on the equivalence of pth moment stability between SDDEs and their numerical methods are stated and proved, which covers a large class of SDDEs and various numerical methods.
\item The constraint on the step size of truncated EM method for SDDEs is released, which makes the method more applicable.
\end{itemize}
\par
This paper is constructed in the following way. The mathematical preliminaries are presented in Section \ref{sec:mathpre}. Section \ref{sec:mainiff} sees the general theorem of the equivalence of the pth moment stability. Section \ref{sec:tEM} is devoted to the study on the truncated EM method. Numerical examples are conducted to demonstrate the theoretical results in Section \ref{sec:numsim}. Section \ref{sec:conl} concludes this paper by emphasizing the main contributions of this work.

\section{Mathematical preliminaries} \label{sec:mathpre}
Throughout this paper we use the following notations. Let $|\cdot|$ be the Euclidean norm in $\R^{n}$ and $\langle x,y \rangle$ be the inner product of vectors $x,y\in \R ^{n}$. If A is a vector or matrix, its transpose is denoted by $A^T$. If A is a matrix, its trace norm is denoted by $|A|=\sqrt{trace(A^TA)}$. If $x$ is a real number, its integer part is denoted by $In[x]$. Let $\R_+=[0,\infty)$ and $\tau >0$. Let $C([-\tau ,0];\R^n)$ denote the family of continuous functions $\varphi$ from $[-\tau ,0]$ to $\R^n$.

Let $(\Omega,\mathscr{F},\{\mathscr{F}_t \}_{t\geq 0},\P)$ be a complete probability space with a filtration $\{\mathscr{F}_t\}_{t\geq 0}$ satisfying the usual conditions(i.e.,it is right continuous and $\mathscr{F}_0$ contains all $\P$-null sets). Let $\omega(t)=(\omega_1(t),\cdots ,\omega_m(t))^T$ be an m-dimensional Brownian motion defined on the probability space. Moreover, for two real numbers a and b, we use $ a\vee b=max(a,b)$ and $a\wedge b=min(a,b)$. If G is a set, its indicator function is denoted by $\mathbbm{1}_G$, namely $\mathbbm{1}_G(x)=1$ if $x\in G$ and $0$ otherwise. Denote by $L^p_{\mathscr{F}_{t}}([-\tau ,0];\R^n)$ the family of $\mathscr{F}_t$-measurable, $C([-\tau ,0];\R^n)$ valued random variables $\xi=\{\xi(u):-\tau\leq u\leq 0\}$ such that
\begin{equation*}
||\xi||^p:=\sup_{-\tau\leq u\leq0}|\xi(u)|^p<\infty.
\end{equation*}
If $y(t)$ is a continuous $\R^n$-valued stochastic process on $t\in [-\tau ,\infty)$, we let $y_t=\{y(t+u):-\tau \leq u\leq 0\}$ for $t\geq 0$ which is regarded as a $C([-\tau ,0];\R^n)$ valued stochastic process.

Let us consider the n-dimensional autonomous stochastic delay differential equations (SDDEs)
\begin{equation}
\label{SDDE}
dy(t)=f(y(t),y(t-\tau))dt+g(y(t),y(t-\tau))d\omega(t),
\end{equation}
with the initial data $\xi\in L^p_{\mathscr{F}_0}([-\tau ,0];\R^n)$, where $f:\R^n \times \R^n\rightarrow \R^n$ and $g:\R^n\times \R^n\rightarrow \R^{n\times m}$.
\par
Denote the numerical solution to the SDDE \eqref{SDDE} by $x(t;0,\xi)$, whose detailed structure is not needed in Sections \ref{sec:mathpre} and \ref{sec:mainiff}. Our definition of the exponential stability in $p$th moment is as follows.
\begin{definition}
The solution to the SDDE (\ref{SDDE}) is said to be exponentially stable in $p$th moment for any $p > 0$, if there is a pair of positive constants $\lambda$ and $M$ such that, for any initial data $\xi\in L^p_{\mathscr{F}_0}([-\tau ,0];\R^n)$
\begin{equation}
\label{wending}
\E|y(t;0,\xi)|^p\leq M||\xi||^pe^{-\lambda t}\quad\forall t\geq0.
\end{equation}
\end{definition}

In this paper we often need to introduce the solution to the SDDE (\ref{SDDE}) for initial data $y_s=\xi\in L^p_{\mathscr{F}_s}([-\tau ,0]; \R^n)$ give at time t=s. As long as the existence and uniqueness of this solution denoted by  $y(t;s,\xi)$ on $t\geq s-\tau$ is guaranteed. It is easy to observe that the solutions to the SDDE (\ref{SDDE}) have the following flow property:
\begin{equation*}
y(t;0,\xi)=y(t;s,y_s)\quad\forall 0\leq s<t<\infty .
\end{equation*}
Moreover, due to the autonomous property of the SDDE (\ref{SDDE}), the exponential stability (\ref{wending}) implies
\begin{equation*}
\E|y(t;s,\xi)|^p\leq M||\xi||^pe^{-\lambda(t-s)}\quad\forall t\geq s.
\end{equation*}

\begin{definition}
A numerical method to the SDDE (\ref{SDDE})is said to be exponentially stable in the $pth$ moment for any $p > 0$ if there is a pair of positive constants $\gamma$ and $H$ such that with initial data $\xi\in L^p_{\mathscr{F}_0}([-\tau ,0]; \R^n)$,
\begin{equation*}
\E|x(t;0,\xi)|^p\leq H||\xi||^pe^{-\gamma t}\quad\forall t\geq0.
\end{equation*}
\end{definition}
\par \noindent
The next two assumptions are needed for Theorem \ref{t1}. Briefly speaking, Assumption \ref{a1} needs that the underlying and numerical solutions have the finite $p$th moment and Assumption \ref{a2} requires that the numerical solution converges to the underlying solution in a finite time with any convergence rate.
\begin{assumption}
\label{a1}
The underlying solution and the numerical solution to SDDE (\ref{SDDE}) satisfy
\begin{equation}
\label{jiexichuzhi}
\sup_{-\tau\leq t\leq \tau}\E|y(t;0,\xi)|^p\leq C^*||\xi||^p,
\end{equation}
and
\begin{equation*}
\sup_{-\tau \leq t \leq \tau}\E|x(t;0,\xi)|^p\leq C^*||\xi||^p,
\end{equation*}
respectively, where $C^*$ is a constant independent of $\xi$.
\end{assumption}

\begin{assumption}
\label{a2}
Write $y(t;0,\xi)=y(t)$ and define $x(t)=x(t;\tau ,y_\tau)$ which is the numerical solution to the SDDE (\ref{SDDE}) with initial data $y_\tau$ starting from $t=\tau$, then
\begin{equation}
\label{qiangshoulian}
\sup_{\tau \leq t\leq \tau +T}\E|x(t)-y(t)|^p\leq C(T)||\xi||^p\alpha(\Delta),
\end{equation}
where $C(T)$ depends on T but not on $\xi$ and $\Delta$ and $\alpha(\Delta)$ is an increasing function with respect to $\Delta$.
\end{assumption}

\section{A general theorem} \label{sec:mainiff}
To prove the main theorem, we present two lemmas firstly.

\begin{lemma}
\label{l1}
Let Assumptions \ref{a1} and \ref{a2} hold. For any $p > 0$, suppose that the SDDE (\ref{SDDE}) is exponentially stable in $pth$ moment, namely
\begin{equation*}
\E|y(t;0,\xi)|^p\leq M||\xi||^pe^{-\lambda t}\quad\forall t\geq 0,
\end{equation*}
for all $\xi\in L^p_{\mathscr{F}_0}([-\tau ,0];\R^n)$. Then there exists a $\Delta_1\geq 0$ such that for every $\Delta\leq\Delta_1$, the numerical solution to  the SDDE (\ref{SDDE}) is exponentially stable in $pth$ moment with rate constant $\gamma$ and growth constant $H$, both of which are independent of $\Delta$. More precisely,
\begin{equation*}
\E|x(t;0,\xi)|^p\leq H||\xi||^pe^{-\gamma t}\quad\forall t\geq 0,
\end{equation*}
with $\gamma=\frac{1}{2}\lambda$ and $H=2^{p+1}MC^*e^{\frac{1}{2}\lambda T}$, where
\begin{equation*}
T=\tau(9+In[4\log(2^pM)/\lambda\tau]).
\end{equation*}
\end{lemma}
\begin{proof}
Fix any initial data $\xi$, write $x(t;0,\xi)=x(t)$ and define $y(t)=y(t;\tau,x_\tau)$. The exponential stability in the $pth$ moment if the SDDE (\ref{SDDE}) shows
\begin{equation}
\label{T}
\E|y(t)|^p\leq M||x_\tau||^pe^{-\tau(t-\tau)}\quad \text{for}~\forall t\geq \tau.
\end{equation}
By the definition of T, we observe that
\begin{equation}
\label{T2}
2^pMe^{-\lambda(T-2\tau)}\leq e^{-\frac{3}{4}\lambda T}.
\end{equation}
By the elementary inequality
\begin{equation*}
(a+b)^p\leq (2(a\vee b)^p)\leq 2^p(a^p\vee b^p)\leq 2^p(a^p+b^p)\quad\forall a,b\geq 0,
\end{equation*}
we have
\begin{equation}
\label{3}
\E|x(t)|^p\leq 2^p(\E|x(t)|-\E|y(t)|^p+\E|y(t)|^p).
\end{equation}
By (\ref{qiangshoulian}) and (\ref{T})
\begin{equation*}
\begin{split}
\sup_{T-\tau \leq t\leq 2T-\tau}\E|x(t)|^p&\leq 2^p(C(2T-2\tau)||\xi||^p\alpha(\Delta)+M||x_{\tau}||^p e^{-\lambda(T-2\tau)})\\
&\leq[2^pC(2T-2\tau)\alpha(\Delta)+2^pMe^{-\lambda(T-2\tau)}]\sup_{-\tau\leq t\leq \tau}\E|x(t)|^p,
\end{split}
\end{equation*}
if necessary, let $\Delta_1$ be even smaller $(\Delta\leq \Delta_1)$ so that $$2^pC(2T-2\tau)\alpha(\Delta)+2^pMe^{-\lambda(T-2\tau)}\leq e^{(-\frac{1}{2}\lambda T)}.$$
Then
\begin{equation}
\label{+-}
\sup_{T-\tau\leq t \leq 2T-\tau}\E|x(t)|^p \leq e^{-\frac{1}{2}\lambda T}\sup_{-\tau \leq t \leq \tau}\E|x(t)|^p.
\end{equation}
Recall that T is a multiple of $\tau$ and hence of $\Delta$. So, by the flow property, for any integer $i\geq 0$
$$x(t)=x(t;iT,x_{iT})\quad\forall t\geq iT.$$
Repeating the argument above for $x(t;iT,x_{iT})$ in the same way that (\ref{+-}) was obtained we may establish
\begin{equation*}
\sup_{(i+1)T-\tau\leq t\leq (i+2)T-\tau}\E|x(t)|^p\leq e^{-\frac{1}{2}\lambda T}\sup_{iT-\tau\leq t\leq iT+\tau}\E|x(t)|^p.
\end{equation*}
From this we see that
\begin{equation*}
\sup_{(i+1)T-\tau\leq t\leq (i+2)T-\tau}\E|x(t)|^p\leq e^{-\frac{1}{2}\lambda T}\sup_{iT-\tau\leq t\leq (i+1)T-\tau}\E|x(t)|^p.
\end{equation*}
By iteration, we can get that
\begin{equation}
\label{guina}
\sup_{(i+1)T-\tau\leq t\leq (i+2)T-\tau}\E|x(t)|^p\leq e^{-\frac{1}{2}\lambda(i+1)T}\sup_{-\tau\leq t\leq T-\tau}\E|x(t)|^p.
\end{equation}
Now, by (\ref{qiangshoulian}) and (\ref{T}), we see from (\ref{3}) that
\begin{equation*}
\begin{split}
\sup_{\tau\leq t\leq T-\tau}\E|x(t)|^p&\leq 2^p(\E|x(t)-y(t)|^p+\E|y(t)|^p)\\
&\leq2^p(C(T-2\tau)||\xi||^P\alpha(\Delta)+M||x_{\tau}||^p)\\
&\leq(2^pC(T-2\tau)\alpha(\Delta)+2^pM))\sup_{-\tau\leq t\leq\tau}\E|x(t)|^p.
\end{split}
\end{equation*}
Choose $\Delta_2(\Delta\leq\Delta_2)$ then
\begin{equation*}
2^pC(T-2\tau)\alpha(\Delta)+2^pM\leq 2^{p+1}M.
\end{equation*}
Then
\begin{equation*}
\sup_{\tau\leq t\leq T-\tau}\E|x(t)|^p\leq 2^{p+1}M\sup_{-\tau\leq t\leq\tau}\E|x(t)|^p.
\end{equation*}
Substituting this into (\ref{guina}) and bearing in mind that M must not be less than 1 we obtain that
\begin{equation*}
\begin{split}
\sup_{(i+1)T-\tau\leq t\leq (i+2)T-\tau}\E|x(t)|^p &\leq 2^{p+1}Me^{-\frac{1}{2}\lambda (i+1)T}\sup_{-\tau\leq t\leq\tau}\E|x(t)|^p \\
&\leq 2^{p+1}MC^*||\xi||^pe^{-\frac{1}{2}\lambda (i+1)T}\quad\forall i\geq 0,
\end{split}
\end{equation*}
while
\begin{equation*}
\sup_{0\leq t\leq T-\tau}\E|x(t)|^p\leq 2^{p+1}MC^*||\xi||^p.
\end{equation*}
Hence
\begin{equation*}
\E|x(t)|^p\leq 2^{p+1}MC^*e^{\frac{1}{2}\lambda T}||\xi||^p e^{-\frac{1}{2}\lambda t}\quad\forall t\geq 0.
\end{equation*}
That is, the numerical method is exponentially stable in the pth moment with $\gamma=\frac{1}{2}\lambda$ and $H=2^{p+1}MC^*e^{\frac{1}{2}\lambda T}$. This completes the proof.
\end{proof}
Lemma \ref{l1} shows that the expenential stable in the pth moment of the SDDE(\ref{SDDE}) implies the exponential stability in the pth moment of the numerical method for small $\Delta$. Let us now eatablish the converse theorem.

\begin{lemma}
\label{l2}
Let Assumption \ref{a2} hold. Assume that for some $\Delta\geq 0$, the numerical method is exponentially stable in the pth moment on the SDDE (\ref{SDDE}), namely
\begin{equation*}
\E|x(t;0,\xi)|^p\leq H||\xi||^pe^{-\gamma t}\quad\forall t\geq 0,
\end{equation*}
for all $\xi\in L^p_{\mathscr{F}_0}([-\tau ,0];\R^n)$. Suppose we can verify
\begin{equation}
\label{4}
2^pC(2T-2\tau)\alpha(\Delta)+2^pHe^{-\gamma(T-2\tau)}\leq e^{(-\frac{1}{2}\lambda T)},
\end{equation}
where $T=\tau(9+In[4\log(2^pH)/\gamma\tau])$. Then the SDDE is exponentially stable in the pth moment. More precisely,
\begin{equation*}
\E|y(t;0,\xi)|^p\leq M\E||\xi||^pe^{-\lambda t}\quad\forall t\geq 0,
\end{equation*}
with
\begin{equation*}
\lambda=\frac{1}{2}\gamma~~and~~M=2^{p+1}HC^*e^{\frac{1}{2}\gamma T}.
\end{equation*}
\end{lemma}
The proof of this lemma is proved in the same way as lemma \ref{l1} was proved, so we only give the outline but highlight the different part.

\begin{proof}
Fix any initial data $\xi$, write $y(t;0,\xi)=y(t)$ and define $x(t)=x(t;\tau,y_\tau)$. The exponential stability in the pth moment of the numerical method shows
\begin{equation}
\label{wending2}
\E|x(t)|^p\leq H||y_\tau||^pe^{-\gamma(t-\tau)}\quad\text{for}~\forall T\geq\tau.
\end{equation}
By (\ref{qiangshoulian}) and (\ref{wending2}) we can show that
\begin{equation*}
\begin{split}
\E|y(t)|^p&\leq 2^p(\E|x(t)-y(t)|^p+\E|x(t)|^p)\\
\sup_{T-\tau \leq t\leq 2T-\tau}\E|y(t)|^p&\leq 2^p(C(2T-2\tau)||\xi||^p\alpha(\Delta)+H||y_{\tau}||^pe^{-\gamma(T-2\tau)}).
\end{split}
\end{equation*}
Using (\ref{4}) we get that
\begin{equation*}
\sup_{T-2\tau\leq t\leq 2T-\tau}\E|y(t)|^p\leq e^{-\frac{1}{2}\gamma T}\sup_{-\tau\leq t\leq \tau}\E|y(t)|^p.
\end{equation*}
Repeating thie argument we find that
\begin{equation}
\label{guinahou}
\sup_{(i+1)T-\tau\leq t\leq (i+2)T-\tau}\E|y(t)|^p \leq e^{-\frac{1}{2}\gamma (i+1)T}\sup_{-\tau\leq t\leq T-\tau}\E|y(t)|^p.
\end{equation}
On the other hand,by (\ref{qiangshoulian}) and (\ref{wending2}), we can show that
\begin{equation*}
\sup_{\tau\leq t\leq T-\tau}\E|y(t)|^p\leq (2^pC(T-2\tau)\alpha(\Delta)+2^pH\sup_{-\tau\leq t\leq \tau}\E|y(t)|^p.
\end{equation*}
This together with (\ref{jiexichuzhi}), yields
\begin{equation}
\label{guinahou2}
\sup_{-\tau\leq t\leq T-\tau}\E|y(t)|^p\leq C^*(2^pC(T-2\tau)\alpha(\Delta)+2^pH)||\xi||^p.
\end{equation}
It then follows from (\ref{guinahou}) and (\ref{guinahou2}) that
\begin{equation*}
\E|y(t)|^p\leq C^*2^{p+1}He^{\frac{1}{2}\gamma T}e^{-\frac{1}{2}\gamma t}\quad\forall t\geq 0,
\end{equation*}
as required.
\end{proof}

Now, we are ready for the main theorem.

\begin{theorem}
\label{t1}
Under Assumption \ref{a1} and Assumption \ref{a2}, the SDDE (\ref{SDDE}) is exponentially stable in pth moment if and only if for some $\Delta>0$, the numerical method is exponentially stable in pth moment.
\end{theorem}

Combining lemmas \ref{l1} and \ref{l2} we obtain the following sufficient and necessary theorem.

\begin{remark}
The reason that we regard Theorem \ref{t1} as a general result is due to the assumptions we made, where only the finite time convergence and the moment boundedness of the numerical method in a very short time are needed, but no particular structure of the method is specified.
\end{remark}

\section{The truncated EM method} \label{sec:tEM}
This section is to show that the truncated EM method is a numerical approximation whose $p$th moment exponential stability is equivalent to that of the underlying SDDEs. To achieve this, we show the finite time strong convergence as well as the moment boundedness of the method firstly. Then the application of Theorem \ref{t1} implies the desired result.

\subsection{Brief introduction}
To make this paper self contained, we brief the truncated EM method for SDDEs in this part, along which we also discuss the fact that the requirement on the step size is weaken in this paper.
\par

For the SDDE \ref{SDDE}, we impose following assumptions on the drift and diffusion coefficients.
\begin{assumption}
\label{assumption1}
Assume that the coefficients f and g satisfy the local Lipschitz condition: For any $R>0$, there is a $K_R>0$ such that
\begin{equation*}
\begin{split}
&|f(x,y)-f(\bar{x},\bar{y})|\vee|g(x,y)-g(\bar{x},\bar{y})|\leq K_R(|x-\bar{x}|+|y-\bar{y}|)\\
&|x|\vee|y|\vee|\bar{x}|\vee|\bar{y}|\leq R,
\end{split}
\end{equation*}
for all $x,y,\bar{x},\bar{y}\in\R^n.$
\end{assumption}
\begin{assumption}
\label{assumption2}
Assume that the coefficients satisfy the Khasminskii-type condition: There is a pair of constants $p>2$ and $K_1>0$ such that
\begin{equation*}
x^Tf(x,y)+\frac{p-1}{2}|g(x,y)|^2\leq K_1(1+|x|^2+|y|^2),
\end{equation*}
for all $x,y\in\R^n.$
\end{assumption}
\begin{assumption}
\label{assumption3}
We need an additional condition. To state it, we need a new notation. Let $\mathcal{U}$ denote the family of continuous functions $U:\R^n\times\R^n\rightarrow\R_+$ such that for each $b>0,$ there is a positive constant $K_b$ for which
$$U(x,\bar{x})\leq K_b|x-\bar{x}|^2\quad\forall x,\bar{x}\in\R^n,~with~|x|\vee|\bar{x}|\leq b.$$
Assume that there is a pair of constants $q>2$ and $H_1>0$ such that
\begin{equation*}
\begin{split}
&(x-\bar{x})^T(f(x,y)-f(\bar{x},\bar{y}))+\frac{q-1}{2}|g(x,y)-g(\bar{x},\bar{y})|^2\\
&\leq H_1(|x-\bar{x}|^2+|y-\bar{y}|^2)-U(x,\bar{x})+U(y,\bar{y}),
\end{split}
\end{equation*}
for all $x,y,\bar{x},\bar{y}\in\R^n.$
\end{assumption}
\begin{assumption}
\label{assumption4}
Assume that there is a pair of positive constants $\rho$ and $H_2$ such that
\begin{equation*}
\begin{split}
&|f(x,y)-f(\bar{x},\bar{y})|^2\vee|g(x,y)-g(\bar{x},\bar{y})|^2\\
&\leq H_2(1+|x|^\rho+|y|^\rho+|\bar{x}|^\rho+|\bar{y}|^\rho)(|x-\bar{x}|^2+|y-\bar{y}|^2),
\end{split}
\end{equation*}
for all $x,y,\bar{x},\bar{y}\in\R^n.$
\end{assumption}
\begin{assumption}
\label{assumption5}
There is a pair of constants $K_2>0$ and $\gamma\in(0,1]$ such that the initial data $\xi$ satisfies
\begin{equation*}
|\xi(u)-\xi(v)|\leq K_2|u-v|^\gamma, \quad -\tau\leq v< u\leq 0.
\end{equation*}
\end{assumption}
To define the truncated EM numerical solutions, we first choose a strictly increasing continuous function $\mu :\R_+\rightarrow\R_+$ such that $\mu(r)\rightarrow\infty$ as $r\rightarrow\infty$ and
\begin{equation*}
\sup_{|x|\vee|y|\leq r}(|f(x,y)|\vee|g(x,y)|)\leq \mu(r),\quad \forall r\geq1.
\end{equation*}
Denoted by $\mu^{-1}$ is the inverse function of $\mu$ and we see that $\mu^{-1}$ is a strictly increasing continuous function from $[\mu(1),\infty)$ to $\R_+$. We also choose a constant $\hat{h}\geq1\vee\mu(1)$ and a strictly decreasing function $h:(0,1]\rightarrow[\mu(1),\infty)$ such that
\begin{equation*}
\lim_{\Delta\rightarrow0}h(\Delta)=\infty~~and~~\Delta^{1/4}h(\Delta)\leq\hat{h},\quad \forall\Delta\in(0,1].
\end{equation*}
We will later that Assumption \ref{assumption4} implies (\ref{muquzhiyiju}), namely that both coefficients f and g grow at most polynomially, hence we can let $\mu(u)=H_3u^{1+0.5\rho}$, where $H_3$ is a positive constant specified in (\ref{muquzhiyiju}). Moreover, we can let $h(\Delta)=\hat{h}\Delta^{-\varepsilon}$ for some $\varepsilon\in(0,1/4]$. In other words, there are lots of choices for $\mu(\cdot)$ and $h(\cdot)$.
\par
A comparison of assumptions of the method in this paper and those in the previous work is presented in Appendix \ref{sec:appen}.

For a given step size $\Delta\in(0,1]$, let us define a mapping $\pi_{\Delta}$ from $\R^n$ to the closed ball $x\in\R^n$ by
$$\pi_{\Delta}(x)=(|x|\wedge\mu^{-1}(h(\Delta)))\frac{x}{|x|},$$
where we set $x/|x|=0$ when $x=0$. That is, $\pi_{\Delta}$ will map $x$ to itself when $|x|\leq\mu^{-1}(h(\Delta))$ and to $\mu^{-1}(h(\Delta))/|x|$ when $|x|>\mu^{-1}(h(\Delta))$. We then define the truncated functions
\begin{equation*}
f_{\Delta}(x,y)=f(\pi_{\Delta}(x),\pi_{\Delta}(y))~~and~~g_{\Delta}(x,y)=g(\pi_{\Delta}(x),\pi_{\Delta}(y)),
\end{equation*}
for $x,y\in\R^n$. It is easy to see that
\begin{equation}
\label{fgbeikongzhi}
|f_{\Delta}(x,y)|\vee|g_{\Delta}(x,y)|\leq\mu(\mu^{-1}(h(\Delta)))=h(\Delta),\quad\forall x,y\in\R^n.
\end{equation}
\par
From now on, we will let the step size $\Delta$ be a fraction of $\tau.$ That is, we will use $\Delta=\tau/M$ for some positive integer M. When we use the terms of a sufficiently small $\Delta,$ we mean that we choose M sufficiently large.
\par
Let us now form the discrete-time truncated EM solutions. Define $t_k=k\Delta$ for $k=-M,-(M-1),\cdots,0,1,2,\cdots.$ Set $X_k=\xi(t_k)$ for $k=-M,-(M-1),\cdots,0$ and then form
\begin{equation*}
X_{k+1}=X_k+f_{\Delta}(X_k,X_{k-M})\Delta+g_{\Delta}(X_k,X_{k-M})\Delta W_k,
\end{equation*}
for $k=0,1,2,\cdots,$ where $\Delta W_k=W(t_{k+1})-W(t_k).$ In our analysis, it is more convenient to work on the continuous-time step process $\bar{x}(t)$ on $t\in[-\tau,\infty)$ defined by
\begin{equation*}
\bar{x}(t)=\sum_{k=-M}^{\infty}X_k\mathbbm{1}_{[k\Delta,(k+1)\Delta)}(t),
\end{equation*}
where $\mathbbm{1}_{[k\Delta,(k+1)\Delta)}(t)$ is the indicator function of $[k\Delta,(k+1)\Delta)$(please recall the notation defined in the beginning of this paper). The other one is the continuous-time continuous process $x(t)$ on $t\in[-\tau,\infty)$ defined by $x(t)=\xi(t)$ for $t\in[-\tau,0]$ while for $t\geq 0$
\begin{equation*}
x(t)=\xi(0)+\int^{t}_{0}f_{\Delta}(\bar{x}(s),\bar{x}(s-\tau))ds+\int_{0}^{t}g_{\Delta}(\bar{x}(s),\bar{x}(s-\tau))dW(s).
\end{equation*}
We see that $x(t)$ is an It\^o process on $t\geq 0$ with its It\^o differential
\begin{equation}
\label{shuzhiweifen}
dx(t)=f_{\Delta}(\bar{x}(s),\bar{x}(s-\tau))dt+g_{\Delta}(\bar{x}(s),\bar{x}(s-\tau))dW(t).
\end{equation}
It is useful to know that $X_k=\bar{x}(t_k)=x(t_k)$ for every $k\geq-M$, namely they coincide at $t_k.$ Of course, $\bar{x}(t)$ is computable but $x(t)$ is not in general.

\subsection{Results}

\begin{lemma}
\label{lemma1}
Let Assumption \ref{assumption2} hold. Then, for all $\Delta\in(0,1]$, we have
\begin{equation}
\label{lemma1gongshi}
x^Tf_{\Delta}(x,y)+\frac{p-1}{2}|g_{\Delta}(x,y)|^2\leq\hat{k}(1+|x|^2+|y|^2)\quad \forall x,y\in\R^n,
\end{equation}
where $\hat{k}=2K_1(1\vee\frac{1}{\mu^{-1}(h(1))})$.
\end{lemma}
\begin{proof}
Fix any $\Delta\in(0,1].$ For $x\in\R^n$ with $|x|\leq\mu^{-1}(h(\Delta))$ and any $y\in \R^n$,by Assumption \ref{assumption2} we have,
\begin{equation*}
\begin{split}
&\quad x^Tf_{\Delta}(x,y)+\frac{p-1}{2}|g_{\Delta}(x,y)|^2\\
&=\pi_{\Delta}(x)^Tf(\pi_{\Delta}(x),\pi_{\Delta}(y))+\frac{p-1}{2}|g(\pi_{\Delta}(x),\pi_{\Delta}(y))|^2\\
&\leq K_1(1+|\pi_{\Delta}(x)|^2+|\pi_{\Delta}(y)|^2)\\
&\leq K_1(1+|x|^2+|y|^2),
\end{split}
\end{equation*}
which implies the desired assertion (\ref{lemma1gongshi}). On the other hand, for $x\in \R^n$ with $|x|>\mu^{-1}(h(\Delta))$ and any $y\in\R^n,$by Assumption \ref{assumption2}, we have
\begin{equation*}
\begin{split}
&\quad x^Tf_{\Delta}(x,y)+\frac{p-1}{2}|g_{\Delta}(x,y)|^2\\
&=\pi_{\Delta}(x)^Tf(\pi_{\Delta}(x),\pi_{\Delta}(y))+\frac{p-1}{2}|g(\pi_{\Delta}(x),\pi_{\Delta}(y))|^2+(x-\pi_{\Delta}(x))^Tf(\pi_{\Delta}(x),\pi_{\Delta}(y))\\
&\leq K_1(1+|\pi_{\Delta}(x)|^2+|\pi_{\Delta}(y)|^2)+(x-\pi_{\Delta}(x))^Tf(\pi_{\Delta}(x),\pi_{\Delta}(y))\\
&\leq K_1(1+|\pi_{\Delta}(x)|^2+|\pi_{\Delta}(y)|^2)+(\frac{|x|}{\mu^{-1}(h(\Delta))}-1)\pi_{\Delta}(x)^Tf(\pi_{\Delta}(x),\pi_{\Delta}(y))\\
&\leq\frac{|x|}{\mu^{-1}(h(\Delta))}K_1(1+|\pi_{\Delta}(x)|^2+|\pi_{\Delta}(y)|^2)\\
&\leq K_1(1\vee\frac{1}{\mu^{-1}(h(\Delta))})(|x|+|x|^2+|x||y|),
\end{split}
\end{equation*}
by the element inequality $ab\leq(a^2+b^2)/2$, we can get that
\begin{equation*}
\begin{split}
&\quad K_1(1\vee\frac{1}{\mu^{-1}(h(\Delta))})(|x|+|x|^2+|x||y|)\\
&\leq 2K_1(1\vee\frac{1}{\mu^{-1}(h(\Delta))})(1+|x|^2+|y|^2)\\
&\leq \hat{k}(1+|x|^2+|y|^2),
\end{split}
\end{equation*}
where $\hat{k}=2K_1(1\vee\frac{|x|}{\mu^{-1}(h(\Delta))}).$
\end{proof}
The following lemma is useful to observe that the truncated functions $f_\Delta$ and $g_\Delta$ preserve assumption \ref{assumption4} perfectly.
\begin{lemma}
Let Assumption \ref{assumption4} hold. Then, for all $\Delta\in(0,1],$ we have
\begin{equation*}
\begin{split}
&|f_\Delta(x,y)-f_\Delta(\bar{x},\bar{y})|^2\vee|g_\Delta(x,y)-g_\Delta(\bar{x},\bar{y})|^2\\
&\leq H_2(1+|x|^\rho+|y|^\rho+|\bar{x}|^\rho+|\bar{y}|^\rho)(|x-\bar{x}|^2+|y-\bar{y}|^2),
\end{split}
\end{equation*}
for all $x,y,\bar{x},\bar{y}\in\R^n.$
\end{lemma}
\begin{proof}
\begin{equation*}
\begin{split}
&\quad|f_\Delta(x,y)-f_\Delta(\bar{x},\bar{y})|^2\vee|g_\Delta(x,y)-g_\Delta(\bar{x},\bar{y})|^2\\
&=|f(\pi_{\Delta}(x),\pi_{\Delta}(y))-f(\pi_{\Delta}(\bar{x}),\pi_{\Delta}(\bar{y}))|^2\vee|g(\pi_{\Delta}(x),\pi_{\Delta}(y))-g(\pi_{\Delta}(\bar{x}),\pi_{\Delta}(\bar{y}))|^2\\
&\leq H_2(1+|\pi_{\Delta}(x)|^\rho+|\pi_{\Delta}(y)|^\rho+|\pi_{\Delta}(\bar{x})|^\rho+|\pi_{\Delta}(\bar{y})|^\rho)(|\pi_{\Delta}(x)-\pi_{\Delta}(\bar{x})|^2+|\pi_{\Delta}(y)-\pi_{\Delta}(\bar{y})|^2),
\end{split}
\end{equation*}
for all $x,y,\bar{x},\bar{y}\in\R^n.$ Noting
\begin{equation*}
\begin{split}
&|\pi_{\Delta}(x)|\leq |x|,~~|\pi_{\Delta}(y)|\leq|y|,~~|\pi_{\Delta}(\bar{x})|\leq|\bar{x}|,~~|\pi_{\Delta}(\bar{y})|\leq|\bar{y}|,\\
&|\pi_{\Delta}(x)-\pi_{\Delta}(\bar{x})|^2\leq|x-\bar{x}|^2,~~|\pi_{\Delta}(y)-\pi_{\Delta}(\bar{y})|^2\leq|y-\bar{y}|^2,
\end{split}
\end{equation*}
we get
\begin{equation*}
\begin{split}
&|f_\Delta(x,y)-f_\Delta(\bar{x},\bar{y})|^2\vee|g_\Delta(x,y)-g_\Delta(\bar{x},\bar{y})|^2\\
&\leq H_2(1+|x|^\rho+|y|^\rho+|\bar{x}|^\rho+|\bar{y}|^\rho)(|x-\bar{x}|^2+|y-\bar{y}|^2).
\end{split}
\end{equation*}
\end{proof}
Moreover, we also observe from Assumption \ref{assumption4} that
\begin{equation}
\label{muquzhiyiju}
|f(x,y)|\vee|g(x,y)|\leq H_3(|x|\vee|y|)^{(\rho+2)/2},\quad\forall |x|,|y|\geq1,
\end{equation}
where $H_3=\sqrt{6H_2}+|f(0,0)|+|g(0,0)|.$
\begin{lemma}
\label{lemma3}
For any $\Delta\in(0,1]$ and any $\hat{p}>0,$ we have
\begin{equation*}
\E|x(t)-\bar{x}(t)|^{\hat{p}}\leq C_{\hat{p}}\Delta^{\hat{p}/2}(h(\Delta))^{\hat{p}},\quad \forall t\geq 0,
\end{equation*}
where $C_{\hat{p}}$ is a positive constant dependent only on $\hat{p}.$ Consequently
\begin{equation*}
\lim_{\Delta\rightarrow 0}E|x(t)-\bar{x}(t)|^{\hat{p}}=0,\quad \forall t\geq 0.
\end{equation*}
In what follows, we will use $C_{\hat{p}}$ to stand for generic positive real constants dependent only on $\hat{p}$ and its values may change between occurrences. Fix $\Delta\in(0,1]$ arbitrarily. For any $t\geq 0$, there is a unique integer $k\geq 0$ such that $t_{k}\leq t<t_{k+1}$.
\par
It should be pointed out that this lemma is need to be proved only for $\hat{p}\geq 2$. However, it is easy to see that this lemma holds for any $\hat{p}\in(0,2)$ as well. In fact, by the H\"older inequality, for any $\hat{p}\in(0,2)$, we have
\begin{equation*}
\E|x(t)-\bar{x}(t)|^{\hat{p}}\leq(\E|x(t)-\bar{x}(t)|^2)^{\hat{p}/2}\leq(C_2\Delta(h(\Delta))^2)^{\hat{p}/2}=C_{\hat{p}}\Delta^{\hat{p}/2}(h(\Delta))^{\hat{p}}.
\end{equation*}
\begin{proof}
When $\hat{p}\geq2$, by (\ref{fgbeikongzhi}) as well as the H\"older inequality and the moment property of the It\^o integral, we can derive from (\ref{shuzhiweifen}) that
\begin{equation*}
\begin{split}
\E|x(t)-\bar{x}(t)|^{\hat{p}}&=\E|x(t)-x(t_{k})|^{\hat{p}}\\
&\leq C_{\hat{p}}(\E|\int^{t}_{t_k}f_{\Delta}(\bar{x}(s),\bar{x}(s-\tau))ds|^{\hat{p}}+\E|\int^{t}_{t_k}g_{\Delta}(\bar{x}(s),\bar{x}(s-\tau))dw(s)|^{\hat{p}})\\
&\leq C_{\hat{p}}(\Delta^{\hat{p}-1}\E\int^{t}_{t_k}|f_{\Delta}(\bar{x}(s),\bar{x}(s-\tau))|^{\hat{p}})ds+\Delta^{(\hat{p}-2)/2}\E\int^{t}_{t_k}|g_\Delta(\bar{x}(s),\bar{x}(s-\tau))|^{\hat{p}}ds\\
&\leq C_{\hat{p}}\Delta^{\frac{\hat{p}}{2}}(h(\Delta))^{\hat{p}}.
\end{split}
\end{equation*}
\end{proof}
\end{lemma}
\begin{lemma}
\label{lemma4}
Let Assumption \ref{assumption1} and \ref{assumption2} hold. Then,
\begin{equation}
\label{lemma4gongshi}
\sup_{0\leq\Delta\leq1}(\sup_{0\leq t\leq T}\E|x(t)|^p)\leq C,\quad \forall T>0.
\end{equation}
where, and from now on,  C stands for generic positive real constants dependent on $T,\xi$ etc. but independent of $\Delta$ and its values may change between occurrences.
\begin{proof}
Fix any $\Delta\in(0,1].$ By the It\^o formula, we derive from (\ref{shuzhiweifen}) that, for $0\leq t\leq T,$
\begin{equation*}
\begin{split}
\E|x(t)|^p&\leq|\xi(0)|^p+\E\int^{t}_{0}p|x(s)|^{p-2}\\
&\quad\times(x^{T}(s)f_{\Delta}(\bar{x}(s),\bar{x}(s-\tau))+\frac{p-1}{2}|g_{\Delta}(\bar{x}(s),\bar{x}(s-\tau))|^2)ds\\
&=|\xi(0)|^{p}+\E\int^{t}_{0}p|x(s)|^{p-2}\\
&\quad\times(\bar{x}^{T}(s)f_{\Delta}(\bar{x}(s),\bar{x}(s-\tau))+\frac{p-1}{2}|g_{\Delta}(\bar{x}(s),\bar{x}(s-\tau))|^2)ds\\
&\quad+\E\int^{t}_{0}p|x(s)|^{p-2}(x(s)-\bar{x}(s))^{T}f_{\Delta}(\bar{x}(s),\bar{x}(s-\tau))ds.
\end{split}
\end{equation*}
Noting from the Young inequality that
\begin{equation*}
a^{p-2}b\leq\frac{p-2}{p}a^{p}+\frac{2}{p}b^{p/2},\quad \forall a,b\geq 0,
\end{equation*}
as well as using lemma \ref{lemma1}, we then have
\begin{equation*}
\begin{split}
\E|x(t)^p|&\leq|\xi(0)|^p+\E\int^{t}_{0}p\hat{k}|x(s)|^{p-2}(1+|\bar{x}(s)|^2+|\bar{x}(s-\tau)|^2)ds\\
&\quad+(p-2)\E\int^{t}_{0}|x(s)|^pds\\
&\quad+2\E\int^{t}_{0}|x(s)-\bar{x}(s)|^{p/2}|f_{\Delta}(\bar{x}(s),\bar{x}(s-\tau))|^{p/2}ds\\
&\leq C+C\int^{t}_{0}(\E|x(s)|^p+\E|\bar{x}(s)|^p+\E|\bar{x}(s-\tau)|^p)ds\\
&\quad+2\E\int^{T}_{0}|x(s)-\bar{x}(s)|^{p/2}|f_{\Delta}(\bar{x}(s),\bar{x}(s-\tau))|^{p/2}ds.
\end{split}
\end{equation*}
But, by Lemma \ref{lemma3} with $p=\hat{p}$ and inequalities (\ref{fgbeikongzhi}), we have
\begin{equation*}
\begin{split}
&\quad \E\int^{T}_{0}|x(s)-\bar{x}(s)|^{p/2}|f_{\Delta}(\bar{x}(s),\bar{x}(s-\tau))|^{p/2}ds\\
&\leq(h(\Delta))^{p/2}\int^{T}_{0}\E(|x(s)-\bar{x}(s)|^{p/2})ds\\
&\leq C_{p}T(h(\Delta))^{p}\Delta^{p/4}\\
&\leq C_pT.
\end{split}
\end{equation*}
We therefore have
\begin{equation*}
\begin{split}
\E|x(t)|^p&\leq C+C\int^{t}_{0}(\E|x(s)|^p+\E|\bar{x}(s)|^p+\E|x(s-\tau)|^p)ds\\
&\leq C+C\int^{t}_{0}(\E|x(s)|^p+\E|\bar{x}(s)|^p+\E|x(s)|^p)ds+\int^{0}_{-\tau}\E|x(s)|^pds\\
&\leq C+C\int^{t}_{0}(\sup_{0\leq u\leq s}\E|x(u)|^p)ds.
\end{split}
\end{equation*}
As this holds for any $t\in[0,T]$ while the sum of the right hand side terms is non-decreasing in t, we then see
\begin{equation*}
\sup_{0\leq u\leq t}\E|x(u)|^p\leq C+C\int^{t}_{0}(\sup_{0\leq u\leq s}\E|x(u)|^p)ds.
\end{equation*}
The well-known Gronwall inequality yields that
\begin{equation*}
\sup_{0\leq u\leq T}\E|x(u)|^p\leq C.
\end{equation*}
As this holds for any $\Delta\in(0,1]$ while C is independent of $\Delta$, we see the required assertion (\ref{lemma4gongshi}).
\end{proof}
\end{lemma}
\begin{theorem}
\label{thm:strotEM}
Let Assumptions \ref{assumption1}, \ref{assumption2}, \ref{assumption3}, \ref{assumption4} and \ref{assumption5} hold with $2p>(2+\rho)q$. Then, for any $\bar{q}\in[2,q)$ and $\Delta\in(0,1],$
\begin{equation}
\label{theorem1}
\E|y(T)-x(T)|^{\bar{q}}\leq C\bigg((\mu^{-1}(h(\Delta)))^{-(2p-(2+\rho)\bar{q})/2}+\Delta^{\bar{q}/2}(h(\Delta))^{\bar{q}}+\Delta^{\bar{q}\gamma}\bigg)
\end{equation}
and
\begin{equation}
\label{theorem2}
\E|y(T)-\bar{x}(T)|^{\bar{q}}\leq C\bigg((\mu^{-1}(h(\Delta)))^{-(2p-(2+\rho)\bar{q})/2}+\Delta^{\bar{q}/2}(h(\Delta))^{\bar{q}}+\Delta^{\bar{q}\gamma}\bigg).
\end{equation}

\begin{proof}
Fix $\bar{q}\in[2,q)$ and $\Delta\in(0,1]$ arbitrarily. Let $e_{\Delta}(t)=y(t)-x(t)$ for $t\geq0.$ For each integer $n>|x_0|$, define the stopping time
\begin{equation*}
\theta_{n}=\inf\{t\geq0:|y(t)|\vee|x(t)|\geq n\},
\end{equation*}
where we set inf$\emptyset=\infty$(as usual $\emptyset$ denotes the empty set). By the It\^o formula, we have that for any $0\leq t\leq T,$
\begin{equation}
\label{wucha}
\begin{split}
\E|e_{\Delta}(t\wedge\theta_{n})|^{\bar{q}}&=\E\int^{t\wedge\theta_{n}}_{0}\bar{q}|e_{\Delta}(s)|^{\bar{q}-2}\bigg(e_\Delta^T(s)[f(y(s),y(s-\tau))-f_{\Delta}(\bar{x}(s),\bar{x}(s-\tau))]\\
&\quad+\frac{\bar{q}-1}{2}|g(y(s),y(s-\tau))-g_{\Delta}(\bar{x}(s),\bar{x}(s-\tau))|^2\bigg)ds.
\end{split}
\end{equation}
By elementary inequality $2ab=2\sqrt{\varepsilon}a\frac{1}{\sqrt{\varepsilon}}b\leq\varepsilon a^2+\frac{1}{\varepsilon}b^2$ with $ \varepsilon=\frac{q-\bar{q}}{\bar{q}-1}$, we have
\begin{equation*}
\begin{split}
&\quad\frac{\bar{q}-1}{2}|g(y(s),y(s-\tau))-g_{\Delta}(\bar{x}(s),\bar{x}(s-\tau))|^2\\
&\leq\frac{\bar{q}-1}{2}\bigg((1+\frac{q-\bar{q}}{\bar{q}-1})|g(y(s),y(s-\tau))-g(x(s),x(s-\tau))|^2\\
&\quad+(1+\frac{\bar{q}-1}{q-\bar{q}})|g(x(s),x(s-\tau))-g_{\Delta}(\bar{x}(s),\bar{x}(s-\tau))|^2\bigg)\\
&=\frac{q-1}{2}|g(y(s),y(s-\tau))-g(x(s),x(s-\tau))|^2\\
&\quad+\frac{(\bar{q}-1)(q-1)}{2(q-\bar{q})}|g(x(s),x(s-\tau))-g_{\Delta}(\bar{x}(s),\bar{x}(s-\tau))|^2,
\end{split}
\end{equation*}
we get from (\ref{wucha}) that
\begin{equation}
\label{j1+j2}
\E|e_{\Delta}(t\wedge\theta_{n})|^{\bar{q}}\leq J_1+J_2,
\end{equation}
where
\begin{equation*}
\begin{split}
J_1&=\E\int^{t\wedge\theta_{n}}_{0}\bar{q}|e_{\Delta}(s)|^{\bar{q}-2}\bigg(e_{\Delta}^T(s)[f(y(s),y(s-\tau))-f_{\Delta}(x(s),x(s-\tau))]\\
&\quad+\frac{q-1}{2}|g(y(s),y(s-\tau))-g(x(s),x(s-\tau))|^2\bigg)ds
\end{split}
\end{equation*}
and
\begin{equation*}
\begin{split}
J_2&=\E\int^{t\wedge\theta_{n}}_{0}\bar{q}|e_{\Delta}(s)|^{\bar{q}-2}\bigg(e_{\Delta}^T(s)[f(x(s),x(s-\tau))-f_{\Delta}(\bar{x}(s),\bar{x}(s-\tau))]\\
&\quad+\frac{(\bar{q}-1)(q-1)}{2(q-\bar{q})}|g(x(s),x(s-\tau))-g_{\Delta}(\bar{x}(s),\bar{x}(s-\tau))|^2\bigg)ds.
\end{split}
\end{equation*}
By Assumption \ref{assumption3}, we have
\begin{equation*}
\begin{split}
J_1&\leq\bar{q}H_1\E\int^{t\wedge\theta_{n}}_{0}|e_{\Delta}(s)|^{\bar{q}-2}(|y(s)-x(s)|^2+|y(s-\tau)-x(s-\tau)|^2)ds\\
&\quad+\E\int^{t\wedge\theta_{n}}_{0}(-U(y(s),x(s))+U(y(s-\tau),x(s-\tau)))ds.
\end{split}
\end{equation*}
Moreover, by the property of the $\mathcal{U}$-class function U and Assumption \ref{assumption5}, we have
\begin{equation*}
\begin{split}
&\quad \E\int^{t\wedge\theta_{n}}_{0}(-U(y(s),x(s))+U(y(s-\tau),x(s-\tau)))ds\\
&\leq \E\int^{t\wedge\theta_{n}}_{0}(-U(y(s),x(s)))ds+\E\int^{t\wedge\theta_{n}}_{-\tau}U(y(s),x(s))ds\\
&=\E\int^{0}_{-\tau}U(y(s),x(s))ds\\
&\leq 0.
\end{split}
\end{equation*}
For any $p>2$, noting from the Young inequality that
\begin{equation*}
a^{p-2}b\leq\frac{p-2}{p}a^{p}+\frac{2}{p}b^{p/2},\quad \forall a,b\geq 0,
\end{equation*}
we then have
\begin{equation}
\label{j1}
\begin{split}
J_1&\leq\bar{q}H_1(\E\int^{t\wedge\theta_{n}}_{0}|e_{\Delta}(s)|^{\bar{q}}ds+E\int^{t\wedge\theta_{n}}_{0}|e_{\Delta}(s)|^{\bar{q}-2}|e_{\Delta}(s-\tau)|^2ds)\\
&\leq \bar{q}H_1(\E\int^{t\wedge\theta_{n}}_{0}|e_{\Delta}(s)|^{\bar{q}}ds+\E\int^{t\wedge\theta_{n}}_{0}\frac{\bar{q}-2}{\bar{q}}|e_{\Delta}(s)|^{\bar{q}}ds+\E\int^{t\wedge\theta_{n}}_{0}\frac{2}{q}|e_{\Delta}(s-\tau)|^{\bar{q}}ds)\\
&\leq\bar{q}H_1(\E\int^{t\wedge\theta_{n}}_{0}2|e_\Delta(s)|^{\bar{q}}ds+\E\int^{0}_{-\tau}\frac{2}{\bar{q}}|e_{\Delta}(s)|^{\bar{q}}ds)\\
&\leq 2\bar{q}H_1\E\int^{t\wedge\theta_{n}}_{0}|e_{\Delta}(s)|^{\bar{q}}ds.
\end{split}
\end{equation}
Rearranging $J_2$, we get
\begin{equation}
\label{j21+j22}
J_2\leq J_{21}+J_{22},
\end{equation}
where
\begin{equation}
\begin{split}
J_{21}=&\E\int^{t\wedge\theta_{n}}_{0}\bar{q}|e_{\Delta}(s)|^{\bar{q}-2}\bigg(e_{\Delta}^T(s)[f(x(s),x(s-\tau))-f_{\Delta}(x(s),x(s-\tau))]\\
&+\frac{(\bar{q}-1)(q-1)}{(q-\bar{q})}|g(x(s),x(s-\tau))-g_{\Delta}(x(s),x(s-\tau))|^2\bigg)ds\\
J_{22}=&\E\int^{t\wedge\theta_{n}}_{0}\bar{q}|e_{\Delta}(s)|^{\bar{q}-2}\bigg(e_{\Delta}^T(s)[f_{\Delta}(x(s),x(s-\tau))-f_{\Delta}(\bar{x}(s),\bar{x}(s-\tau))]\\
&+\frac{(\bar{q}-1)(q-1)}{(q-\bar{q})}|g_{\Delta}(x(s),x(s-\tau))-g_{\Delta}(\bar{x}(s),\bar{x}(s-\tau))|^2\bigg)ds.
\end{split}
\end{equation}
We estimate $J_{21}$ first. By the Young inequality for any $a,b\geq 0$ and $0\leq t\wedge\theta_n\leq t\leq T,$ we can show that
\begin{equation}
\begin{split}
\label{j21}
J_{21}&\leq \E\int^{t\wedge\theta_{n}}_{0}\bar{q}|e_{\Delta}(s)|^{\bar{q}-2}\bigg(0.5|e_{\Delta}(s)|^2+0.5|f(x(s),x(s-\tau))-f_{\Delta}(x(s),x(s-\tau))|^2\\
&\quad+\frac{(\bar{q}-1)(q-1)}{(q-\bar{q})}|g(x(s),x(s-\tau))-g_{\Delta}(x(s),x(s-\tau))|^2\bigg)ds\\
&\leq \frac{(\bar{q}-1)^2(q-2)}{(q-\bar{q})}\E\int^{t\wedge\theta_{n}}_{0}|e_{\Delta}(s)|^{\bar{q}}ds+\E\int^{t\wedge\theta_{n}}_{0}|f(x(s),x(s-\tau))-f_{\Delta}(x(s),x(s-\tau))|^{\bar{q}}ds\\
&\quad+\frac{2(\bar{q}-1)(q-1)}{(q-\bar{q})}|g(x(s),x(s-\tau))-g_{\Delta}(x(s),x(s-\tau))|^{\bar{q}}ds\\
&\leq C_1\E\int^{t\wedge\theta_{n}}_{0}|e_{\Delta}(s)|^{\bar{q}}ds+J_{23},
\end{split}
\end{equation}
where
\begin{equation*}
J_{23}=\E\int^{t\wedge\theta_{n}}_{0}|f(x(s),x(s-\tau))-f_{\Delta}(x(s),x(s-\tau))|^{\bar{q}}+|g(x(s),x(s-\tau))-g_{\Delta}(x(s),x(s-\tau))|^{\bar{q}}ds
\end{equation*}
and
\begin{equation*}
C_1=max\{\frac{(\bar{q}-1)^2(q-2)}{(q-\bar{q})},1,\frac{2(\bar{q}-1)(q-1)}{(q-\bar{q})}\}.
\end{equation*}
Due to $t\wedge\theta_{n}\leq T$ and Assumption \ref{assumption4}, we derive that
\begin{equation*}
\begin{split}
J_{23}&\leq 2\times10^{\bar{q}/2}C_1H_2\E\int^{T}_{0}\bigg(1+|x(s)|^{\rho\bar{q}/2}+|x(s-\tau)|^{\rho\bar{q}/2}+|\pi_\Delta(x(s))|^{\rho\bar{q}/2}+|\pi_\Delta(x(s-\tau))|^{\rho\bar{q}/2}\bigg)\\
&\quad\times\bigg(|x(s)-\pi_\Delta(x(s))|^{\bar{q}}+|x(s-\tau)-\pi_\Delta(x(s-\tau))|^{\bar{q}}\bigg)ds\\
&\leq 4\times10^{\bar{q}/2}C_1H_2\E\int^{T}_{0}\bigg(1+|x(s)|^{\rho\bar{q}/2}+|x(s-\tau)|^{\rho\bar{q}/2}\bigg)\\
&\quad\times\bigg(|x(s)-\pi_\Delta(x(s))|^{\bar{q}}+|x(s-\tau)-\pi_\Delta(x(s-\tau))|^{\bar{q}}\bigg)ds.
\end{split}
\end{equation*}
Using the H$\ddot{o}$dler inequality and lemma \ref{lemma4} yields
\begin{equation*}
\begin{split}
J_{23}&\leq 4\times10^{\bar{q}/2}C_1H_2\int^{T}_{0}\bigg(\E(1+|x(s)|^{p}+|x(s-\tau)|^{p})\bigg)^{\frac{\rho\bar{q}}{2p}}\\
&\quad\times\bigg(\E|x(s)-\pi_\Delta(x(s))|^{\frac{2p\bar{q}}{2p-\rho\bar{q}}}+\E|x(s-\tau)-\pi_\Delta(x(s-\tau))|^{\frac{2p\bar{q}}{2p-\rho\bar{q}}}\bigg)^{\frac{2p-\rho\bar{q}}{2p}}ds\\
&\leq4\times10^{\bar{q}/2}C_1H_2(1+2C)^{\frac{\rho\bar{q}}{2p}}\int^{T}_{0}\bigg(\E|x(s)-\pi_\Delta(x(s))|^{\frac{2p\bar{q}}{2p-\rho\bar{q}}}\\
&\quad+\E|x(s-\tau)-\pi_\Delta(x(s-\tau))|^{\frac{2p\bar{q}}{2p-\rho\bar{q}}}\bigg)^{\frac{2p-\rho\bar{q}}{2p}}ds\\
&\leq4\times10^{\bar{q}/2}C_1H_2(1+2C)^{\frac{\rho\bar{q}}{2p}}\int^{T}_{0}\bigg([P\{|x(s)|>\mu^{-1}(h(\Delta))\}]^{\frac{2p-(2+\rho)\bar{q}}{2p-\rho\bar{q}}}[\E|x(s)|^p]^{\frac{2\bar{q}}{2p-\rho\bar{q}}}\\
&\quad+[P\{|x(s-\tau)|>\mu^{-1}(h(\Delta))\}]^{\frac{2p-(2+\rho)\bar{q}}{2p-\rho\bar{q}}}[\E|x(s-\tau)|^p]^{\frac{2\bar{q}}{2p-\rho\bar{q}}}\bigg)^{\frac{2p-\rho\bar{q}}{2p}}ds\\
&\leq4\times10^{\bar{q}/2}C_1H_2(1+2C)^{\frac{(\rho+2)\bar{q}}{2p}}\int^{T}_{0}(\frac{\E|x(s)|^p}{(\mu^{-1}(h(\Delta)))^p})^{\frac{2p-(2+\rho)\bar{q}}{2p}}+(\frac{\E|x(s-\tau)|^p}{(\mu^{-1}(h(\Delta)))^p})^{\frac{2p-(2+\rho)\bar{q}}{2p}}ds\\
&\leq8\times10^{\bar{q}/2}C_1H_2(1+2C)(\mu^{-1}(h(\Delta)))^{-\frac{2p-(2+\rho)\bar{q}}{2}}.
\end{split}
\end{equation*}
Substituting this into (\ref{j23}) gives
\begin{equation*}
J_{21}\leq C_1\E\int^{t\wedge\theta_{n}}_{0}|e_{\Delta}(s)|^{\bar{q}}ds+8\times10^{\bar{q}/2}C_1H_2(1+2C)(\mu^{-1}(h(\Delta)))^{-\frac{2p-(2+\rho)\bar{q}}{2}}.
\end{equation*}
Similarly, we can show
\begin{equation*}
\begin{split}
J_{22}&\leq C_1\E\int^{t\wedge\theta_{n}}_{0}|e_{\Delta}(s)|^{\bar{q}}ds+C_2\int^{T}_{0}\bigg(\E|x(s)-\bar{x}(s)|^{\frac{2p\bar{q}}{2p-\rho\bar{q}}}+\E|x(s-\tau)-\bar{x}(s-\tau)|^{\frac{2p\bar{q}}{2p-\rho\bar{q}}}\bigg)^{\frac{2p-\rho\bar{q}}{2p}}ds\\
&\leq C_1\E\int^{t\wedge\theta_{n}}_{0}|e_{\Delta}(s)|^{\bar{q}}ds+C_2\bigg(2\int^{T}_{0}\E|x(s)-\bar{x}(s)|^{\bar{q}}ds+\int^{0}_{-\tau}\E|x(s)-\bar{x}(s)|^{\bar{q}}ds\bigg),
\end{split}
\end{equation*}
where $C_2$ and the following $C_3$, etc. are generic constants independent of $\Delta.$ By lemma \ref{lemma3}, we then have
\begin{equation}
\label{j23}
\begin{split}
J_{22}&\leq C_1\E\int^{t\wedge\theta_{n}}_{0}|e_{\Delta}(s)|^{\bar{q}}ds+C_2\bigg(2TC_{\bar{q}}\Delta^{\frac{\bar{q}}{2}}(h(\Delta))^{\bar{q}}+\int^{0}_{-\tau}|\xi(s)-\xi(\,s/\Delta\!)|^{\bar{q}}ds\bigg)\\
&\leq C_1\E\int^{t\wedge\theta_{n}}_{0}|e_{\Delta}(s)|^{\bar{q}}ds+C_3(\Delta^{\frac{\bar{q}}{2}}(h(\Delta))^{\bar{q}}+\Delta^{\bar{q}\gamma}).
\end{split}
\end{equation}
By (\ref{j21+j22})  we get that
\begin{equation}
\label{j2}
\begin{split}
J_2&\leq C_4\E\int^{t\wedge\theta_{n}}_{0}|e_{\Delta}(s)|^{\bar{q}}ds+C_4\bigg(\Delta^{\frac{\bar{q}}{2}}(h(\Delta))^{\bar{q}}+\Delta^{\bar{q}\gamma}+(\mu^{-1}(h(\Delta)))^{-\frac{2p-(2+\rho)\bar{q}}{2}}\bigg).
\end{split}
\end{equation}
Combining (\ref{j1+j2}), (\ref{j1}) and (\ref{j2})together, we get
\begin{equation*}
\begin{split}
\E|e_{\Delta}(t\wedge\theta)|^{\bar{q}}&\leq C_5\E\int^{t\wedge\theta_{n}}_{0}|e_{\Delta}(s)|^{\bar{q}}ds+C_5\bigg(\Delta^{\frac{\bar{q}}{2}}(h(\Delta))^{\bar{q}}+\Delta^{\bar{q}\gamma}+(\mu^{-1}(h(\Delta)))^{-\frac{2p-(2+\rho)\bar{q}}{2}}\bigg)\\
&\leq C_5\E\int^{t}_{0}\E|e_{\Delta}(s\wedge\theta_{n})|^{\bar{q}}ds+C_5\bigg(\Delta^{\frac{\bar{q}}{2}}(h(\Delta))^{\bar{q}}+\Delta^{\bar{q}\gamma}+(\mu^{-1}(h(\Delta)))^{-\frac{2p-(2+\rho)\bar{q}}{2}}\bigg).
\end{split}
\end{equation*}
An application of the Gronwall inequality yields that
\begin{equation*}
\E|e_{\Delta}(T\wedge\theta_{n})|^{\bar{q}}\leq C\bigg(\Delta^{\frac{\bar{q}}{2}}(h(\Delta))^{\bar{q}}+\Delta^{\bar{q}\gamma}+(\mu^{-1}(h(\Delta)))^{-\frac{2p-(2+\rho)\bar{q}}{2}}\bigg).
\end{equation*}
Using the well-known Fatou Lemma, we can let $n\rightarrow\infty$ to obtain the desired assertion (\ref{theorem1}). The other assertion (\ref{theorem2}) follows from (\ref{theorem1}) and lemma \ref{lemma3}. The proof is therefore complete.
\end{proof}
\end{theorem}

The next corollary provides a clearer view of the convergence rate of the truncated EM method.

\begin{corollary}
In particular, recalling (\ref{muquzhiyiju}) we may define
\begin{equation}
\label{definemu}
\mu(u)=H_3u^{(2+\rho)/2},\quad u\geq1,
\end{equation}
and let
\begin{equation}
\label{defineh}
h(\Delta)=\Delta^{-\varepsilon}\quad for~some~\varepsilon\in(0,1/4]~and~\hat{h}\geq1,
\end{equation}
te get
\begin{equation}
\label{theorem3}
\E|y(T)-x(T)|^{\bar{q}}\leq O(\Delta^{[\varepsilon(2p-(2+\rho)\bar{q})/(2+\rho)]\wedge[\bar{q}(1-2\varepsilon)/2]\wedge[\bar{q}\gamma]})
\end{equation}
and
\begin{equation}
\label{theorem4}
\E|y(T)-\bar{x}(T)|^{\bar{q}}\leq O(\Delta^{[\varepsilon(2p-(2+\rho)\bar{q})/(2+\rho)]\wedge[\bar{q}(1-2\varepsilon)/2]\wedge[\bar{q}\gamma]}).
\end{equation}
When $\mu$ is defined by (\ref{definemu}), then $\mu^{-1}(u)=(u/H_3)^{2/(2+\rho)}.$ Substituting this and (\ref{defineh}) into (\ref{theorem1}) we get
\begin{equation*}
\E|y(T)-x(T)|^{\bar{q}}\leq C(\Delta^{\varepsilon(2p-(2+\rho)\bar{q})/(2+\rho)}+\Delta^{\bar{q}(1-2\varepsilon)/2}+\Delta^{\bar{q}\gamma)},
\end{equation*}
which is the required assertion (\ref{theorem3}). Similarly, we can show (\ref{theorem4}).
What's more, choosing p sufficiently large for
\begin{equation*}
\varepsilon(2p-(2+\rho)\bar{q})/(2+\rho)>\bar{q}(1-2\varepsilon)/2,
\end{equation*}
we can get the following assertions from (\ref{theorem3}) and (\ref{theorem4}) easily.
\begin{equation*}
\E|y(T)-x(T)|^{\bar{q}}\leq O(\Delta^{(\bar{q}(1-2\varepsilon)/2)\wedge (\bar{q}\gamma)})
\end{equation*}
and
\begin{equation*}
\E|y(T)-\bar{x}(T)|^{\bar{q}}\leq O(\Delta^{(\bar{q}(1-2\varepsilon)/2)\wedge (\bar{q}\gamma)}).
\end{equation*}
This corollary shows that if $\gamma$ is close to 0.5 (or bigger than half), this shows that the order of convergence is close to 0.5.
\end{corollary}

We end up this section by Theorem \ref{coro:pthstab}, which illustrates that the truncated EM method is within the scope of Theorem \ref{t1}.

\begin{theorem} \label{coro:pthstab}
Assume $f(0,0) = g(0,0) = 0$. In addition, suppose Assumptions \ref{assumption1}, \ref{assumption2}, \ref{assumption3}, \ref{assumption4} and \ref{assumption5} hold with $2p>(2+\rho)q$. Then, the $p$th moment exponentially stability of the truncated EM method is equivalent to that of the underlying SDDEs.
\end{theorem}
The proof of this theorem is a straight application of Theorem \ref{thm:strotEM} and Theorem \ref{t1} by noting that with $f(0,0) = g(0,0) = 0$ all the constant term on the right hand side in Assumption \ref{assumption2} is dismissed.

\section{Numerical simulations}\label{sec:numsim}
A two-dimensional SDDE considered in this part to demonstrate the theoretical results proved in previous sections.
\begin{example}
Consider a two-dimensional SDDE
\begin{equation}
\label{sddeexample}
\left\{
\begin{array}{lr}
dx_1(t)=(|x_2(t-\tau)|^{4/3}-x_1^{3}(t))dt+(|x_1(t)|^{3/2}+x_2(t-\tau))dw(t),&\\
dx_2(t)=(|x_1(t-\tau)|^{4/3}-x_2^{3}(t))dt+(|x_2(t)|^{3/2}+x_1(t-\tau))dw(t),&
\end{array}
\right.
\end{equation}
We can therefore regard (\ref{sddeexample}) as an SDDE in $\R^2$ with the coefficients
\begin{equation*}
f(x,y)=
\begin{pmatrix}
|y_2|^{4/3}-x_1^3\\
|y_1|^{4/3}-x_2^3
\end{pmatrix}
\end{equation*}
and
\begin{equation*}
g(x,y)=
\begin{pmatrix}
|x_1|^{3/2}+y_2\\
|x_2|^{3/2}+y_1
\end{pmatrix}
\end{equation*}
for $x,y\in\R^2.$
\\
It is obvious that these coefficients are locally Lipschitz continuous, namely, they satisfy Assumption \ref{assumption1}. We also assume that the initial data satisfy Assumption \ref{assumption5}. We then derive
\begin{equation*}
\begin{split}
&x^{T}f(x,y)+\frac{p-1}{2}|g(x,y)|^2\\
&\leq |x_1||y_2|^{4/3}-x_1^4+|x_2||y_1|^{4/3}-x_2^4\\
&+\frac{p-1}{2}\big(|x_1|^3+2|x_1|^{3/2}y_2+|y_2|^2+|x_2|^3+2|x_2|^{3/2}y_1+|y_1|^2\big).
\end{split}
\end{equation*}
But, by the Young inequality,
\begin{equation*}
|x||y|^{4/3}=(|x|^3)^{1/3}(|y|^2)^{2/3}\leq|x|^3+|y|^2,
\end{equation*}
we therefore have
\begin{equation*}
\begin{split}
&\quad x^{T}f(x,y)+\frac{p-1}{2}|g(x,y)|^2\\
&\leq |x_1|^3+|y_2|^2-x_1^4+|x_2|^3+|y_1|^2-x_2^4+(p-1)|x_1|^3+(p-1)|x_2|^3+(p-1)|y_1|^2+(p-1)|y_2|^2\\
&\leq a_2(1+|y|^2),
\end{split}
\end{equation*}
where $a_2=(2a_1\vee p)$ and $a_1=\sup(-u^4+pu^3)$.
That is, Assumption \ref{assumption2} is satisfied as well.
Then, consider Assumption \ref{assumption3}. For $x, y, \bar{x}, \bar{y} \in \R$, it is easy to show that
\begin{equation}
\label{assumption2yanzheng}
\begin{split}
&\quad(x-\bar{x})^T(f(x,y)-f(\bar{x},\bar{y})))\\
&=(x_1-\bar{x}_1)(|y_2|^{4/3}-|\bar{y}_2|^{4/3}-x_1^3+\bar{x}_1^3)+(x_2-\bar{x}_2)(|y_1|^{4/3}-|\bar{y}_1|^{4/3}-x_2^3+\bar{x}_2^3)\\
&=(x_1-\bar{x}_1)(|y_2|^{4/3}-|\bar{y}_2|^{4/3})-(x_1-\bar{x}_1)^2(x_1^2+x_1\bar{x}_1+\bar{x}_1^2)\\
&\quad+(x_2-\bar{x}_2)(|y_1|^{4/3}-|\bar{y}_1|^{4/3})-(x_2-\bar{x}_2)^2(x_2^2+x_2\bar{x}_2+\bar{x}_2^2)\\
&\leq\frac{1}{2}(x_1-\bar{x}_1)^2+\frac{1}{2}(|y_2|^{4/3}-|\bar{y}_2|^{4/3})^2-\frac{1}{2}(x_1-\bar{x}_1)^2(x_1^2+\bar{x}_1^2)\\
&\quad+\frac{1}{2}(x_2-\bar{x}_2)^2+\frac{1}{2}(|y_1|^{4/3}-|\bar{y}_1|^{4/3})^2-\frac{1}{2}(x_2-\bar{x}_2)^2(x_2^2+\bar{x}_2^2).
\end{split}
\end{equation}
But, by the mean value theorem,
\begin{equation*}
(|y|^{4/3}-|\bar{y}|^{4/3})^2\leq\frac{16}{9}|y-\bar{y}|^2(|y|^{1/3}+|\bar{y}|^{1/3})^2\leq4|y-\bar{y}|^2(|y|^{2/3}+|\bar{y}|^{2/3}).
\end{equation*}
Let $a_3=\sup(8u^{2/3}-0.5u^2)$. Then
\begin{equation*}
(|y|^{4/3}-|\bar{y}^{4/3}|)^2\leq a_3|y-\bar{y}|^2+0.25|y-\bar{y}|^2(y^2+\bar{y}^2).
\end{equation*}
Substituting this into (\ref{assumption2yanzheng}) yields
\begin{equation*}
\begin{split}
&\quad(x-\bar{x})^T(f(x,y)-f(\bar{x},\bar{y})))\\
&\leq (x_1-\bar{x_1})^2+a_3|y_2-\bar{y_2}|^2+\frac{1}{4}|y_2-\bar{y_2}|^2(y_2^2+\bar{y_2}^2)-\frac{1}{2}(x_1-\bar{x_1})^2(x_1^2+\bar{x_1}^2)\\
&\quad+(x_2-\bar{x_2})^2+a_3|y_1-\bar{y_1}|^2+\frac{1}{4}|y_1-\bar{y_1}|^2(y_1^2+\bar{y_1}^2)-\frac{1}{2}(x_2-\bar{x_2})^2(x_2^2+\bar{x_2}^2).
\end{split}
\end{equation*}
Similarly, we can show that
\begin{equation*}
\begin{split}
&\quad\frac{q-1}{2}|g(x,y)-g(\bar{x},\bar{y})|^2\\
&\leq a_4(x_1-\bar{x}_1)^2+\frac{1}{4}(x_1-\bar{x}_1)^2(x_1^2+\bar{x}_1^2)+(q-1)|y_2-\bar{y}_2|^2\\
&\quad+a_4(x_2-\bar{x}_2)^2+\frac{1}{4}(x_2-\bar{x}_2)^2(x_2^2+\bar{x}_2^2)+(q-1)|y_1-\bar{y}_1|^2,
\end{split}
\end{equation*}
where $a_4=\sup9(q-1)u-0.5u^2$, then
\begin{equation*}
\begin{split}
&\quad(x-\bar{x})^T(f(x,y)-f(\bar{x},\bar{y})))+\frac{q-1}{2}|g(x,y)-g(\bar{x},\bar{y})|^2\\
&\leq a_5\bigg((x_1-\bar{x}_1)^2+(x_2-\bar{x}_2)^2+(y_1-\bar{y}_1)^2+(y_2-\bar{y}_2)^2\bigg)\\
&\quad-\frac{1}{4}\bigg((x_1-\bar{x}_1)^2(x_1^2+\bar{x}_1^2)+(x_2-\bar{x}_2)^2(x_2^2+\bar{x}_2^2)\bigg)\\
&\quad+\frac{1}{4}\bigg((y_1-\bar{y}_1)^2(y_1^2+\bar{y}_1^2)+(y_2-\bar{y}_2)^2(y_2^2+\bar{y}_2^2)\bigg),
\end{split}
\end{equation*}\
where $a_5=(1\vee a_3\vee a_4 \vee (q-1))$
\\
and $U(x,\bar{x})=\frac{1}{4}\bigg((x_1-\bar{x}_1)^2(x_1^2+\bar{x}_1^2)+(x_2-\bar{x}_2)^2(x_2^2+\bar{x}_2^2)\bigg)$.
\\
 It is obvious that $U\in\mathcal{U}$. In other words, we have shown that Assumption \ref{assumption3} is satisfied too.
Finally, consider Assumption \ref{assumption4}, by the mean value theorem and elementary inequality we can have the following alternative estimate
\begin{equation}
\label{f-f}
\begin{split}
&\quad|f(x,y)-f(\bar{x},\bar{y})|^2\\
&\leq a_7\bigg((y_2-\bar{y}_2)^2(1+y_2^4+\bar{y}_2^4)+(x_1-\bar{x}_1)^2(x_1^4+\bar{x}_1^4)\\
&\quad+(y_1-\bar{y}_1)^2(1+y_1^4+\bar{y}_1^4)+(x_2-\bar{x}_2)^2(x_2^4+\bar{x}_2^4)\bigg)\\
&\leq a_7(1+|y|^4+|\bar{y}|^4+|x|^4+|\bar{x}|^4)(|x-\bar{x}|^2+|y-\bar{y}|^2),
\end{split}
\end{equation}
where $a_7=(\frac{128}{p}a_6\vee 36$ and $a_6=\sup(u^{2/3}-u^4)$.
\\
Similarly,
\begin{equation}
\label{g-g}
\begin{split}
|g(x,y)-g(\bar{x},\bar{y})|^2&\leq a_9(x_1-\bar{x}_1)^2(1+x_1^4+\bar{x}_1^4)+2(y_2-\bar{y}_2)^2\\
&\quad+a_9(x_2-\bar{x}_2)^2(1+x_2^4+\bar{x}_2^4)+2(y_1-\bar{y}_1)^2\\
&\leq 2a_9|x-\bar{x}|^2(1+|x|^4+|\bar{x}|^4+|y|^4+|\bar{y}|^4)\\
&\quad+4|y-\bar{y}|^2(1+|x|^4+|\bar{x}|^4+|y|^4+|\bar{y}|^4)\\
&\leq a_{10}(1+|x|^4+|\bar{x}|^4+|y|^4+|\bar{y}|^4)(|x-\bar{x}|^2+|y-\bar{y}|^2),
\end{split}
\end{equation}
where $a_9=8a_8\vee 4$ and $a_8=\sup(u-u^4)$ and $a_{10}=2a_9\vee 4$
we hence see from (\ref{f-f}) and (\ref{g-g}) that Assumption \ref{assumption4} ia also satisfied.
\par
Let us compute the approximation of the mean square error. We run M=100 independent trajectories for every different step sizes, $10^{-2}$, $10^{-3}$, $10^{-4}$, $10^{-6}$. Because it is hard to find the true solution for the SDDE, the numerical solution with the step size $10^{-6}$ is regard as the exact solution.
\begin{figure}[H]
\centering
\includegraphics[width=0.80\textwidth]{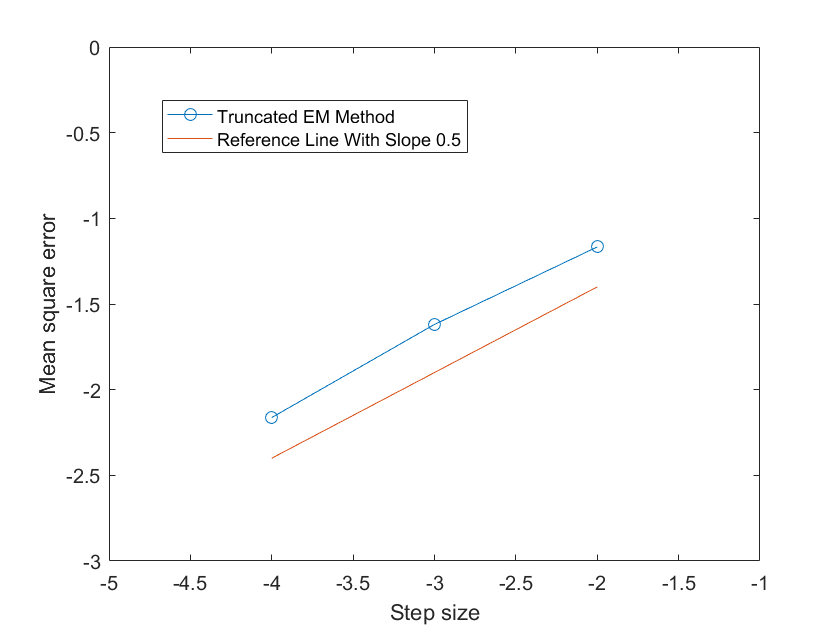}
\caption{Convergence rate of Example}
\end{figure}
Figure 1: The mean square errors between the exact solution and the numerical solutions for step sizes $\Delta=10^{-2},10^{-3},10^{-4}$.
\par
By the linear regression, also shown in the Figure 1, the slope of the errors against the step sizes is approximately 0.498, which is very close to the theoretical result.
\end{example}

\section{Conclusion}\label{sec:conl}
In this paper, we presented a general theorem on the equivalence of the $p$th moment stability between stochastic differential delay equations and their numerical methods. By saying the theorem, we mean that the theorem holds when the numerical methods are strongly convergent and have the bouneded $p$th moment in the finite time while regardless of structures of the numerical methods.
\par
To show that the general theorem is applicable, we study the truncated EM method as an example. Alongside with the investigation of the strong convergence of the truncated EM method, we also significantly release the requirements on the step size compared with the the original work.
\par
Since the $p$ in our setting could be very close to 0, our result also indicates some connections between the almost sure stability of the underlying SDDEs and their numerical methods.

\section*{Appendix} \label{sec:appen}

In \cite{GMY2018}, the truncated EM for SDDE was originally developed. It was required to choose a number $\Delta^*\in(0,1]$ and a strictly decreasing function $h:(0,\Delta^*]\rightarrow[\mu(0),\infty)$ such that
$$h(\Delta^*)\geq\mu(1),~~\lim_{\Delta\rightarrow0}h(\Delta)=\infty~~\text{and}~~\Delta^{1/4}h(\Delta)\leq1,~~\forall\Delta\in(0,\Delta^*].$$
In this paper, we simply let $\Delta^*=1$ and remove the condition that $h(\Delta^*)\geq\mu(1)$. Meanwhile, we also replace requirement that $\Delta^{1/4}h(\Delta)\leq1$ by a weaker one, $\Delta^{1/4}h(\Delta)\leq\hat{h}$. In other words, we have made the choice of function h more flexible. We emphasize that such changes do not make any effect on the results in \cite{GMY2018}. In fact, condition $h(\Delta^*)\geq\mu(1)$ was only used to prove Lemma 4.2 of \cite{GMY2018}. But, in view of Lemma \ref{lemma1} in this paper, we see that the constant $2k_1$ in Lemma 4.2 of \cite{GMY2018} is now replaced by another constant $\hat{k}$ which does not affect any results in \cite{GMY2018}. It is also easy to check that replacing $\Delta^{1/4}h(\Delta)\leq1$ by $\Delta^{1/4}h(\Delta)\leq\hat{h}$ does not make any effect on the other results in \cite{GMY2018}.

\end{document}